\documentclass[11pt]{amsart}

\issueinfo{00}{0}{Month}{Year}
\copyrightinfo{2008}{University of Calgary}
\pagespan{1}{2}
\PII{ISSN 1715-0868}
\usepackage{graphicx}
\usepackage{epstopdf}
\usepackage{amsthm,amsmath,amssymb,amsxtra,amscd}
\usepackage{verbatim}
\usepackage{indent}
\usepackage{rotating}
\usepackage{multirow}
\usepackage{multicol}
\usepackage{url}
\usepackage{algorithmic}
\usepackage{xypic}

\newtheorem{theorem}{Theorem}[section]

\newtheorem{lemma}[theorem]{Lemma}

\newtheorem{definition}[theorem]{Definition}

\newtheorem{proposition}[theorem]{Proposition}

\newtheorem{claim}[theorem]{Claim}

\newtheorem{example}[theorem]{Example}
\newtheorem*{example*}{Example}

\newtheorem{observation}[theorem]{Observation}

\newtheoremstyle{myexample}{3pt}{3pt}{\rmfamily}{}{\itshape}{:}{ }{\thmname{#1}\thmnumber{ #2}\thmnote{ (#3)}}
\theoremstyle{myexample}

\newtheoremstyle{myremark}{3pt}{3pt}{\rmfamily}{}{\itshape}{:}{ }{\thmname{#1}}
\theoremstyle{myremark}
\newtheorem{remark}[theorem]{Remark}
\newtheorem*{observation*}{Observation}

\newtheoremstyle{conjecture}{3pt}{3pt}{\itshape}{}{\bfseries}{.}{ }{\thmname{#1}\thmnote{ (#3)}}
\theoremstyle{conjecture}

\newtheorem*{question*}{Question}

\newtheorem{theorem*}{Theorem}

\numberwithin{equation}{section}

\newcounter{algorithm}
\renewcommand{\thealgorithm}{\thesection.\arabic{algorithm}}

\usepackage{graphicx}
\usepackage{esvect}
\usepackage{amsmath}
\usepackage{amsthm}
\usepackage[hidelinks]{hyperref}
\usepackage{enumitem}
\usepackage{etoolbox}

\def\dom{\mathop{\mathrm{Dom}}\nolimits}

\def\mus{\mathop{\mathrm{mus}}\nolimits}

\def\str#1{\mathbf {#1}}

\def\arity#1{a(\rel{}{#1})}
\def\nbrel#1#2{R\ifstrempty{#1}{}{_{#1}}\ifstrempty{#2}{}{^{#2}}}
\def\rel#1#2{\nbrel{\ifstrempty{#1}{}{\str{#1}}}{#2}}
\def\func#1#2{\nbfunc{\ifstrempty{#1}{}{\str{#1}}}{#2}}
\def\nbfunc#1#2{F\ifstrempty{#1}{}{_{#1}}\ifstrempty{#2}{}{^{#2}}}

\def\K{{\mathcal K}}
\def\Fraisse{Fra\"{\i}ss\' e}

\def\Monoid{{\mathfrak M}}
\def\EMonoid{\Monoid=(M,\oplus,\mleq,0)}
\def\Block{{\mathcal B}}
\def\mleq{\preceq}
\def\mgeq{\succeq}
\def\mgt{\succ}
\def\mlt{\prec}
\def\Str{\mathop{\mathrm{Str}}\nolimits}

\def\Ind#1#2{#1\setbox0=\hbox{$#1x$}\kern\wd0\hbox to 0pt{\hss$#1\mid$\hss}
\lower.9\ht0\hbox to 0pt{\hss$#1\smile$\hss}\kern\wd0}
\def\Notind#1#2{#1\setbox0=\hbox{$#1x$}\kern\wd0\hbox to
0pt{\mathchardef\nn="0236\hss$#1\nn$\kern1.4\wd0\hss}\hbox 
to 0pt{\hss$#1\mid$\hss}\lower.9\ht0
\hbox to 0pt{\hss$#1\smile$\hss}\kern\wd0}

\begin{document}

\title{Conant's generalised metric spaces are Ramsey}

\author[J. Hubi\v cka]{Jan Hubi\v cka}
\address{Charles University, Faculty of Mathematics and Physics\\Department of Applied Mathematics (KAM)\\
Prague, Czech Republic}
\email{hubicka@kam.mff.cuni.cz}
\author[M. Kone\v cn\'y]{Mat\v ej Kone\v cn\'y}
\address{Charles University, Faculty of Mathematics and Physics\\Department of Applied Mathematics (KAM)\\
Prague, Czech Republic}
\email{matej@kam.mff.cuni.cz}
\author[J. Ne\v set\v ril]{Jaroslav Ne\v set\v ril}
\address{Charles University, Faculty of Mathematics and Physics\\
Computer Science Institute of Charles University (IUUK)\\
Prague, Czech Republic}
\email{nesetril@iuuk.mff.cuni.cz}
\thanks {J. H. and M. K. are supported  by  project  18-13685Y  of  the  Czech  Science Foundation (GA\v CR) and by Charles University project Progres Q48.}.\\

\begin{abstract}
We give Ramsey expansions of classes of generalised metric spaces where distances come from a linearly ordered commutative monoid. This complements
results of Conant about the extension property for partial automorphisms and extends 
an earlier result of the first and the last author giving the Ramsey property of convexly ordered $S$-metric spaces.
Unlike Conant's approach, our analysis does not require the monoid to be semi-archimedean.
\end{abstract}
\maketitle
\centerline{Dedicated to old friend Norbert Sauer.}

\section{Introduction}
Given $S\subseteq \mathbb R_{>0}$ (a subset of positive reals), an \emph{$S$-metric space}
is a metric space with all distances contained in $S\cup\{0\}$. 
The following 4-values condition characterises when the class of all finite $S$-metric spaces is closed
for amalgamation (see Section~\ref{sec:background} for definition of amalgamation):
\begin{definition}[Delhomm{\'e}, Laflamme, Pouzet, Sauer \cite{Delhomme2007}]
A subset $S\subseteq \mathbb R_{>0}$ satisfies the \emph{4-values condition}, if
for every $a,b,c,d\in S$, if there is some $x\in S$ such that the triangles with distances $a$--$b$--$x$ and $c$--$d$--$x$ satisfy
the triangle inequality, then there is also $y\in S$ such that the triangles with distances $a$--$c$--$y$ and
$b$--$d$--$y$ satisfy the triangle inequality.
\end{definition}
\begin{figure}
\centering
\includegraphics{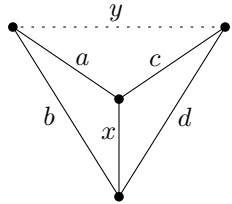}
\caption{The 4-values condition.}
\label{fig:4values}
\end{figure}

The 4-values condition means that the amalgamation of triangles exists (see Figure~\ref{fig:4values}). It can be equivalently described by means of the following operation:
\begin{definition}\label{oplusS}
Given a subset $S$ of positive reals and $a,b\in S$, denote by $a\oplus_S b=\sup\{x\in S; x\leq a+b\}$.
\end{definition}
\begin{theorem}[Sauer~\cite{Sauer2013b}]
\label{thm:sauer2}
A (topologically) closed subset $S$ of the positive reals satisfies the 4-values condition if and only if the operation $\oplus_S$
is associative.
\end{theorem}
Sauer~\cite{Sauer2013} used the equivalence above to determine those $S$ such that there exists an (ultra)homogeneous $S$-metric space (i.e.~the $S$-Urysohn metric space).
In \cite{Hubicka2016} the first and third author proved:
\begin{theorem}[Hubi\v cka, Ne\v set\v ril~\cite{Hubicka2016}]
\label{thm:Smetric}
Given a closed set $S$ of positive reals the following four statements are equivalent:
\begin{enumerate}
 \item The class of all finite $S$-metric spaces is an amalgamation class.
 \item $S$ satisfies 4-values condition.
 \item $\oplus_S$ is associative.
 \item The class of all finite $S$-metric spaces has a precompact Ramsey expansion.
\end{enumerate}
\end{theorem}
This generalises the earlier work on Ramsey property for metric spaces~\cite{Nevsetvril2007,Dellamonica2012,The2010}.
In this paper we further develop this line of research and show similar results in the context of
generalised metric spaces where distance
sets form a monoid as defined by Conant~\cite{Conant2015} and show that the class
of all such generalised metric spaces is Ramsey when enriched by a convex linear ordering (see e.g.~\cite{The2010}).

Let us recall the key notion of a Ramsey class first.
For structures $\str{A},\str{B}$ denote by ${\str{B}\choose \str{A}}$ the set of all substructures of $\str{B}$, which are isomorphic to $\str{A}$ (see Section~\ref{sec:background} for definition of structure and substructure).  Using this notation the definition of a Ramsey class gets the following form:
\begin{definition}
A class $\mathcal C$ is a \emph{Ramsey class} if for every two objects $\str{A}$ and $\str{B}$ in $\mathcal C$ and for every positive integer $k$ there exists a structure $\str{C}$ in $\mathcal C$ such that the following holds: For every partition ${\str{C}\choose \str{A}}$ into $k$ classes there exists an $\widetilde{\str B} \in {\str{C}\choose \str{B}}$ such that ${\widetilde{\str{B}}\choose \str{A}}$ belongs to one class of the partition. 
\end{definition}
 It is usual to shorten the last part of the definition to $\str{C} \longrightarrow (\str{B})^{\str{A}}_k$.

The following is the main result of this paper which complements results of
Conant~\cite{Conant2015} about the extension property for partial automorphisms
(EPPA). Conant generalized results of Solecki~\cite{solecki2005} and
Vershik~\cite{vershik2008} about metric spaces.  Our work extends earlier
Theorem~\ref{thm:Smetric} in a similar manner.

\begin{theorem}\label{thm:main}
For every distance monoid $\EMonoid$ the class $\vv{\mathcal M}_\Monoid$ of all convexly ordered finite $\Monoid$-metric spaces is Ramsey.
\end{theorem}
To define a distance monoid we first recall a standard definition. The definition of convex order will be given in Definition~\ref{defn:Js}.

A {\em commutative monoid} is a triple $(M,\oplus,0)$ where $M$ is a set containing $0$, $\oplus$ is an associative and commutative binary operation on $M$ and $0$ is the identity element of $\oplus$.

The following definition of a distance monoid $\EMonoid$ was given by Conant~\cite{Conant2015}:
\begin{definition}\label{def:DM}
A structure $\EMonoid$ is a \emph{distance monoid} if
\begin{enumerate}
\item $(M,\oplus,0)$ is a commutative monoid with identity $0$;
\item $(M, \mleq, 0)$ is a linear order with least element $0$;
\item For all $a,b,c,d\in M$ it holds that if $a\mleq c$ and $b\mleq d$, then $a\oplus b \mleq c\oplus d$ ($\oplus$ is monotonous in $\mleq$).
\end{enumerate}
\end{definition}
In other words, a distance monoid is a (linearly) positively ordered monoid, see~\cite{wehrung1992}.
For every distance monoid $\Monoid$, one can define a $\Monoid$-metric space.
\begin{definition}\label{def:Mmetric}
Suppose $\EMonoid$ is a distance monoid. Given a set $A$ and a function $d:A\times A\to M$, we call $(A,d)$ an \emph{$\Monoid$-metric space} if
\begin{enumerate}
\item for all $a,b\in A$, $d(a,b)=0$ if and only if $a=b$;
\item for all $a,b\in A$, $d(a,b)=d(b,a)$;
\item for all $a,b,c\in A$, $d(a,c)\mleq d(a,b)\oplus d(b,c)$.
\end{enumerate}
Given a distance monoid $\Monoid$, we let $\mathcal M_\Monoid$ denote the class of finite $\Monoid$-metric spaces.
\end{definition}

\begin{example}
The following are distance monoids:
\begin{enumerate}\label{example1}
\item\label{ex:Smetric} Given a set $S$ of non-negative reals containing 0 and closed under $\oplus_S$ the structure $\Monoid_S=(S,\oplus_S,\leq,0)$ (recall Definition~\ref{oplusS}), where the order $\leq$ is the linear order of reals, forms a distance monoid if and only if operation $\oplus_S$ is associative and thus $S$ satisfies the 4-values condition.
\item\label{infini} Consider the set of non-negative real numbers extended by infinitesimal elements, i.e.~$R^* = \left\{a + b\cdot \mathrm{dx}; a,b\in\mathbb R^+_0\right\}$ with piece-wise addition $+$ and order $\mleq$ given by the standard order of reals and $\mathrm{dx}\mlt a$ for every positive real number $a$. Then $(R^*, +, \mleq, 0)$ is also a distance monoid.
\item The ultrametric $([n], \max, \leq, 0)$, where $[n] = \{0, 1, \ldots, n-1\}$ and $\leq$ is the linear order of integers is a distance monoid.
\end{enumerate}
\end{example}

Given a monoid $(M, \oplus, 0)$, $n\geq 0$ and $r\in M$ we denote by $n\times r$ a summation $r\oplus r\oplus \cdots \oplus r$ of length $n$.

\begin{definition}
A distance monoid $\EMonoid$ is \emph{archimedean} if, for all $r,s\in M$, $r,s\neq 0$, there exists some integer $n>0$ such that $s\mleq n\times r$.
\end{definition}
\begin{example}
Consider the reals extended by infinitesimals as in Example~\ref{example1} (\ref{infini}). This monoid is not archimedean, because $n\times\mathrm{dx} \mlt b$ for every positive real $b$, every integer $n$ and infinitesimal $\mathrm{dx}$.
\end{example}

In Section~\ref{sec:background} we briefly introduce necessary model-theoretic background.
In Section~\ref{sec:multiamalgamation} we review Ramsey classes defined by means of forbidden subconfigurations in the setting of~\cite{Hubicka2016}. Although we deal with metric spaces and thus binary systems, it is useful to formulate it in the context of structures involving both relations and functions which will be used in the proof of our main result.
In Section~\ref{sec:metric} we discuss simple algorithm completing graphs to metric spaces which is essential for our approach.
In Section~\ref{sec:Mmetricarch} we show that the class of finite ordered $\Monoid$-metric spaces is Ramsey for every archimedean monoid $\Monoid$.
Finally, in Section~\ref{sec:Metric} we prove the main result and in Section~\ref{sec:remarks} we discuss future directions of research.

\section{Preliminaries}
\label{sec:background}
We now review some standard model-theoretic notions of structures with relations and functions (see e.g.~\cite{Hodges1993}) with a small variation  that our functions will be partial and symmetric. We follow~\cite{Nevsetvril1976}.

Let $L=L_{\mathcal{R}}\cup L_{\mathcal{F}}$ be a language involving relational symbols $\rel{}{}\in L_{\mathcal{R}}$ and function symbols $F\in L_{\mathcal{F}}$ each having associated arities denoted by $\arity{}>0$ for relations and $d(F)>0, r(F)>0$ for functions. 
An \emph{$L$-structure} is a structure $\str{A}$ with {\em vertex set} $A$, functions 
$\func{A}{}:\dom(\func{A}{})\to {A\choose {r(\func{}{})}}$,
 $\dom(\func{A}{})\subseteq A^{d(\func{}{})}$ for $\func{}{}\in L_F$ and relations $\rel{A}{}\subseteq A^{\arity{}}$ for $\rel{}{}\in L_R$. (Note that  by ${A\choose {r(\func{}{})}}$ we denote, as is usual in this context, the set of all $r(\func{}{})$-element subsets of $A$.)
The set $\dom(\func{A}{})$ is called the {\em domain} of function $\func{}{}$  in $\str{A}$.

Note also that we have chosen to have the range of the function symbols to be the set of subsets (not tuples). This is motivated by \cite{Evans3} where we deal with (Hrushovski) extension properties and we need a ``symmetric'' range. However from the point of view of Ramsey theory this is not an important issue.
 
The language is usually fixed and understood from the context (and it is in most cases denoted by $L$).  If the set $A$ is finite we call $\str A$ a \emph{finite $L$-structure}. We consider only structures with countably many vertices. 
If the language $L$ contains no function symbols, we call $L$ a {\em relational language} and an $L$-structure is also called a \emph{relational $L$-structure}.
Every function symbol $\func{}{}$ such that $d(\func{}{})=1$ is a {\em unary function}. A unary relation of course just defines a subset of elements of $A$.
All functions used in this paper are unary.

\medskip

A \emph{homomorphism} $f:\str{A}\to \str{B}$ is a mapping $f:A\to B$ satisfying for every $\rel{}{}\in L_{\mathcal R}$ and for every $\func{}{}\in L_\mathcal F$ the following two statements:
\begin{enumerate}
\item[(a)] $(x_1,x_2,\ldots, x_{\arity{}})\in \rel{A}{}\implies (f(x_1),f(x_2),\ldots,f(x_{\arity{}}))\in \rel{B}{}$, and,
\item[(b)]  $f(\dom(\func{A}{}))\subseteq \dom(\func{B}{})$
 and $f(\func{A}{}(x_1,x_2,\allowbreak \ldots, x_{d(\func{}{})}))=\func{B}{}(f(x_1),\allowbreak f(x_2),\allowbreak \ldots,\allowbreak f(x_{d(\func{}{})}))$ for every $(x_1,x_2,\allowbreak \ldots, x_{d(\func{}{})})\in \dom(\func{A}{})$.
\end{enumerate} 
For a subset $A'\subseteq A$ we denote by $f(A')$ the set $\{f(x);x\in A'\}$ and by $f(\str{A})$ the homomorphic image of a structure.

 If $f$ is injective, then $f$ is called a \emph{monomorphism}. A monomorphism is called an \emph{embedding} if
for every $\rel{}{}\in L_{\mathcal R}$ and $\func{}{}\in L_\mathcal F$ the following holds:
\begin{enumerate}
\item[(a)] $(x_1,x_2,\ldots, x_{\arity{}})\in \rel{A}{}\iff (f(x_1),f(x_2),\ldots,f(x_{\arity{}}))\in \rel{B}{}$, and,
\item[(b)]  
$(x_1,x_2,\ldots, x_{d(\func{}{})})\in\dom(\func{A}{}) \iff (f(x_1),f(x_2),\ldots,f(x_{d(\func{}{})}))\in \dom(\func{B}{}).$
\end{enumerate}
 
  If $f$ is an embedding which is an inclusion then $\str{A}$ is a \emph{substructure} (or \emph{subobject}) of $\str{B}$.
Observe that for structures with functions it does not hold that every choice of $A\subseteq B$ induces a substructure of $\str{B}$.

 For an embedding $f:\str{A}\to \str{B}$ we say that $\str{A}$ is \emph{isomorphic} to $f(\str{A})$ and $f(\str{A})$ is also called a \emph{copy} of $\str{A}$ in $\str{B}$. Thus $\str{B}\choose \str{A}$ is defined as the set of all copies of $\str{A}$ in $\str{B}$. Finally, $\Str(L)$ denotes the class of all finite $L$-structures and all their embeddings.

\medskip

Let $\str{A}$, $\str{B}_1$ and $\str{B}_2$ be structures with $\alpha_1$ an embedding of $\str{A}$
into $\str{B}_1$ and $\alpha_2$ an embedding of $\str{A}$ into $\str{B}_2$. Then
every structure $\str{C}$
 together with embeddings $\beta_1:\str{B}_1 \to \str{C}$ and
$\beta_2:\str{B}_2\to\str{C}$ satisfying $\beta_1\circ\alpha_1 =
\beta_2\circ\alpha_2$ is called an \emph{amalgamation of $\str{B}_1$ and $\str{B}_2$ over $\str{A}$ with respect to $\alpha_1$ and $\alpha_2$}. 
We will call $\str{C}$ simply an \emph{amalgamation} of $\str{B}_1$ and $\str{B}_2$ over $\str{A}$
(as in most cases $\alpha_1$, $\alpha_2$ and $\beta_1$, $\beta_2$ can be chosen to be inclusion embeddings).

An amalgamation is \emph{strong} if $C=\beta_1(B_1)\cup \beta_2(B_2)$ and moreover $\beta_1(x_1)=\beta_2(x_2)$ if and only if $x_1\in \alpha_1(A)$ and $x_2\in \alpha_2(A)$.
A strong amalgamation is \emph{free}  there are no tuples in any relations  of $\str{C}$ and $\dom(\func{C}{})$, $\func{}{}\in L_\mathcal F$, using both vertices of
$\beta_1(B_1\setminus \alpha_1(A))$ and $\beta_2(B_2\setminus \alpha_2(A))$.
An \emph{amalgamation class} is a class $\K$ of finite structures satisfying the following three conditions:
\begin{enumerate}
\item {\em Hereditary property:} For every $\str{A}\in \K$ and every substructure $\str{B}$ of $\str{A}$ we have $\str{B}\in \K$;
\item {\em Joint embedding property:} For every $\str{A}, \str{B}\in \K$ there exists $\str{C}\in \K$ such that $\str{C}$ contains both $\str{A}$ and $\str{B}$ as substructures;
\item {\em Amalgamation property:} 
For $\str{A},\str{B}_1,\str{B}_2\in \K$ and $\alpha_1$ embedding of $\str{A}$ into $\str{B}_1$, $\alpha_2$ embedding of $\str{A}$ into $\str{B}_2$, there is $\str{C}\in \K$ which is an amalgamation of $\str{B}_1$ and $\str{B}_2$ over $\str{A}$ with respect to $\alpha_1$ and $\alpha_2$.
\end{enumerate}
If the $\str{C}$ in the amalgamation property can always be chosen as the free amalgamation, then $\K$ is a {\em free amalgamation class}.

\section{Previous work --- multiamalgamation} 
\label{sec:multiamalgamation}

We now refine amalgamation classes.
Our aim is to describe strong sufficient criteria for  Ramsey classes. In this paper we follow~\cite{Hubicka2016}.

For $L$ containing a binary relation $\rel{}{\leq}$ we consider the class of all finite $L$-structures $\str{A}$ where the set $A$ is linearly ordered by the relation $\rel{}{\leq}$, of course with all monotone (i.e.~order preserving) embeddings.

\begin{figure}
\centering
\includegraphics{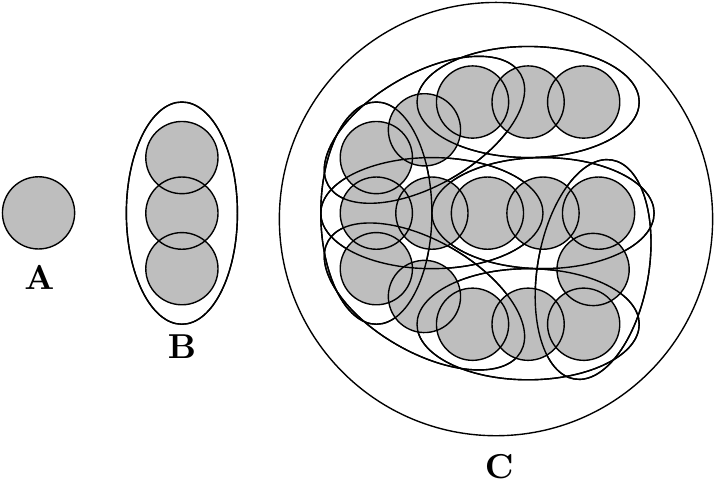}
\caption{Construction of a Ramsey object by multiamalgamation.}
\label{fig:multiamalgam}
\end{figure}
In this setting, we develop a generalised notion of amalgamation which will serve as a useful tool for the construction of Ramsey objects. As schematically depicted in Figure~\ref{fig:multiamalgam}, Ramsey objects are a result of
amalgamation of multiple copies of a given structure which are all performed at once. In a non-trivial class this leads to many problems. We split the amalgamation into two steps---the construction of (up to
isomorphism unique) free amalgamation (which yields an incomplete or ``partial'' structure) followed then by a completion. Formally this is done as follows:

\begin{definition}
\label{def:irreducible}
An $L$-structure $\str{A}$ is \emph{irreducible} if 
$\str{A}$ is not a free amalgamation of two proper substructures of $\str{A}$.
\end{definition}

Thus the irreducibility is meant with respect to the free amalgamation. The irreducible structures are our building blocks.
Moreover in structural Ramsey theory we are fortunate that most structures are (or may be interpreted as) irreducible. And in the most interesting case, the structures may be completed to irreducible structures. 
This will be introduced now by means of the following variant of the homomorphism notion.
\begin{definition}
 A homomorphism
$f:\str{A}\to\str{B}$ is a \emph{homomorphism-em\-bed\-ding}  if $f$ restricted to any irreducible substructure of $\str{A}$      
is an embedding to $\str{B}$.
\end{definition}
While for (undirected) graphs the notions homomorphism and homo\-morphism-em\-bed\-ding coincide, for structures they differ.
For example any homo\-morphism-em\-bedding of the Fano plane into a 3-hypergraph is actually an embedding.

\begin{definition}
\label{defn:completion}
Let $\str{C}$ be a structure. An irreducible structure $\str{C}'$ is a \emph{completion}
of $\str{C}$ if there exists a homomorphism-embedding $\str{C}\to\str{C}'$.
If there is a homomorphism-embedding $\str{C}\to\str{C}'$ which is one-to-one,
we call $\str{C}'$ a \emph{strong completion}.

Let $\str{B}$ be an irreducible substructure of $\str{C}$.  We
say that irreducible structure $\str{C}'$ is a \emph{completion of $\str{C}$ with respect to
copies of $\str{B}$} if there exists a function
$f:C\to C'$ such that for every  $\widetilde{\str{B}}\in
{\str{C}\choose \str{B}}$ the function $f$ restricted to $\widetilde{B}$
is an embedding of $\widetilde{\str{B}}$ to $\str{C}'$.
\end{definition}
We now state all necessary conditions for the main result of \cite{Hubicka2016} which will be used subsequently
(omitting the notion of closure description which is not needed here).
%\begin{definition}
%\label{def:multiamalgamation}
%Let $L$ be a language, $\mathcal R$ be a Ramsey class of finite irreducible $L$-structures.
%We say that a subclass $\mathcal K$ of $\mathcal R$  is an \emph{$\mathcal R$-multi\-amalgamation class} if
%the following conditions are satisfied:
%\begin{enumerate}
 %\item\label{cond:hereditary} {\em Hereditary property:} For every $\str{A}\in \K$ and a substructure $\str{B}$ of $\str{A}$ we have $\str{B}\in \K$.
 %\item\label{cond:amalgamation} {\em Strong amalgamation property:}
%For $\str{A},\str{B}_1,\str{B}_2\in \K$ and $\alpha_1$ embedding of $\str{A}$ into $\str{B}_1$, $\alpha_2$ embedding of $\str{A}$ into $\str{B}_2$, there is $\str{C}\in \K$ which contains a strong amalgamation of $\str{B}_1$ and $\str{B}_2$ over $\str{A}$ with respect to $\alpha_1$ and $\alpha_2$ as a substructure.
 %\item\label{cond:completion} {\em Locally finite completion property:} Let $\str{B}\in \K$ and $\str{C}_0\in \mathcal R$. Then there exists $n=n(\str{B},\str{C}_0)$ such that if an $L$-structure $\str{C}$ satisfies the following:
%\begin{enumerate}
 %\item there is a homomorphism-embedding from $\str{C}$ to $\str{C}_0$,
 %(in other words, $\str{C}_0$ is a completion of $\str{C}$), and,
 %\item every substructure of $\str{C}$ with at most $n$ vertices has a completion to some structure in $\K$.
%\end{enumerate}
%
%Then there exists $\str{C}'\in \K$ that is a completion of $\str{C}$ with respect to copies of $\str{B}$.
%\end{enumerate}
%\end{definition}
\begin{definition}
\label{def:multiamalgamation}
Let $L$ be a language, $\mathcal R$ be a Ramsey class of finite irreducible $L$-structures.% and $\mathcal U$ be a closure description (for $L$).
We say that a subclass $\mathcal K$ of $\mathcal R$  is an \emph{$\mathcal R$-multi\-amal\-ga\-ma\-tion class} if
the following conditions are satisfied:
\begin{enumerate}
 \item\label{cond:hereditary} {\em Hereditary property:} For every $\str{A}\in \K$ and a substructure $\str{B}$ of $\str{A}$ we have $\str{B}\in \K$.
 \item\label{cond:amalgamation} {\em Strong amalgamation property:}
For $\str{A},\str{B}_1,\str{B}_2\in \K$ and embeddings $\alpha_1\colon\str{A}\to\str{B}_1$, $\alpha_2\colon\str{A}\to\str{B}_2$, there is $\str{C}\in \K$ which is a strong amalgamation of $\str{B}_1$ and $\str{B}_2$ over $\str{A}$ with respect to $\alpha_1$ and $\alpha_2$.
 \item\label{cond:completion} {\em Locally finite completion property:} Let $\str{B}\in \K$ and $\str{C}_0\in \mathcal R$. 
Then there exists $n=n(\str{B},\str{C}_0)$ such that if a $L$-structure $\str{C}$ satisfies the following:
\begin{enumerate}
 \item $\str{C}_0$ is a completion of $\str{C}$, 
 \item every irreducible substructure of $\str{C}$ is in $\mathcal K$, and
 \item every substructure of $\str{C}$ with at most $n$ vertices has a $\K$-com\-ple\-tion,
\end{enumerate}
then there exists $\str{C}'\in \K$ which is a completion of $\str{C}$ with respect to copies of $\str{B}$.
\end{enumerate}
\end{definition}

We can now state the main result of~\cite{Hubicka2016} as:
\begin{theorem}[Hubi\v cka, Ne\v set\v ril~\cite{Hubicka2016}]
\label{thm:mainstrongclosures}
Every $\mathcal R$-multiamalgamation class $\K$ is Ramsey.
\end{theorem}

The proof of this result is not easy and involves interplay of several key constructions of 
structural Ramsey theory, particularly Partite Lemma and Partite Construction (see \cite{Hubicka2016} for details). 

We will also make use of the following recent strengthening of Ne\v set\v ril-R\"odl Theorem~\cite{Nevsetvril1976}:
\begin{theorem}[Evans, Hubi\v cka, Ne\v set\v ril~\cite{Evans3}]
\label{thm:NRclosures}
Let $L$ be a language (involving relational symbols and partial functions) and let 
$\K$ be a free amalgamation class of $L$-structures. 
 Then $\vv{\K}$, the class of all structures from $\K$ equipped with an additional linear order $\leq$, is a Ramsey class.
\end{theorem}

\section{Shortest path completion}
\label{sec:metric}
We first show some basic facts about completion to $\Monoid$-metric spaces.
This is similar to~\cite{Hubicka2016} and also to the analysis given in~\cite{Conant2015} proceeds similarly.

Given a distance monoid $\EMonoid$ we interpret an $\Monoid$-metric space as a
relational structure $\str{A}$ in the language $L_\Monoid$ with (possibly infinitely many) binary relations $\rel{}{s}$, $s\in M^{\succ 0}$, where
we put, for every $u\neq v\in A$, $(u,v)\in \rel{A}{s}$ if and only if $d(u,v)=s$. We do not explicitly represent that $d(u,u)=0$ (i.e.~no loops are added). We will also consider ordered $\Monoid$-metric spaces, where the language will further contain a binary relation $\leq$ representing a linear order on the vertices.

\begin{definition}
An \emph{$\Monoid$-graph} is an $L_\Monoid$-structure where all relations are symmetric and irreflexive and every pair of vertices is in at most one relation. (Alternatively, a graph with edges labelled by $M^{\succ 0}$).

An \emph{$\Monoid$-metric graph} us an $\Monoid$-graph which is a non-induced substructure of an $\Monoid$-metric
space (interpreted as $L_\Monoid$-structure) such that all relations are symmetric.

Every $\Monoid$-graph that is not $\Monoid$-metric is a \emph{non-$\Monoid$-metric graph}.
\end{definition}

Observe that $\Monoid$-metric graphs are precisely those structures which have a strong completion to an $\Monoid$-metric space in the sense of Definition~\ref{defn:completion}.

For $\Monoid$-graph $\str{A}$ we will write $d(u,v)=\ell$ if $(a,b)\in \rel{A}{\ell}$. Value of $d(u,v)$ is undefined otherwise.

In the language of $\Monoid$-graphs we will use the following variants of standard graph-theoretic notions.
Given $\Monoid$-graph $\str{A}$ the {\em walk} from $u$ to $v$ is any sequence of vertices $u=v_1,v_2,\ldots, v_n=v\in \str{A}$
such that $d(v_i,v_{i+1})$ is defined for every $1\leq i<n$. If the sequence contains no repeated vertices, it is a \emph{path}.
The $\Monoid$-length of walk (or path) is $d(v_1, v_2)\oplus d(v_2,v_3)\oplus \cdots\oplus d(v_{n-1},v_n)$.
Given vertices $u$ and $v$ the \emph{shortest path from $u$ to $v$} is any
path from $u$ to $v$ such that there is no other path from $u$ to $v$ of strictly smaller $\Monoid$-length (in the order $\mleq$).
We say that $\str{A}$ is \emph{connected} if there exists a path from $u$ to $v$ for every choice of $u\neq v\in A$.

\begin{definition}[Shortest path completion]
Let $\EMonoid$ be a distance monoid and $\str{G}=(G,d)$ be a (finite) connected $\Monoid$-metric graph. For every $u,v\in A$ define $d'(u,v)$ to be the minimum of the $\Monoid$-lengths
of all walks from $u$ to $v$ in $\str{G}$. Then we call the complete $\Monoid$-metric graph $\str{A}=(G,d')$ the {\em
shortest path completion} of $\str{G}$.

Given an $\Monoid$-metric graph we also denote by $\mathcal W(u,v)$ a shortest path connecting $u$ to $v$ such that its $\Monoid$-length is $d'(u,v)$ (there can be multiple of them, in that case pick an arbitrary one).
\end{definition}

The following is the main result of this section.
\begin{proposition}
\label{prop:cycles}
Let $\EMonoid$ be a distance monoid and $\str{G}$ be a finite $\Monoid$-graph.
\begin{enumerate}
\item \label{cor:suaer1:1} If $\str{G}$ is connected and $\Monoid$-metric, then its shortest path completion $\str{A}$ is an $\Monoid$-metric space and it is a strong completion of $\str G$ in the sense of Definition~\ref{defn:completion}.
\item \label{cor:suaer1:2} $\str{G}$ is $\Monoid$-metric if and only if it contains no homomorphic image of a non-$\Monoid$-metric cycle $\str{C}$ (a cycle $v_1,v_2,\ldots,v_n$ such that $d(v_1,v_n)\mgt d(v_1,v_2)\oplus \cdots \oplus d(v_{n-1},v_n))$.
\item \label{cor:suaer1:3} $\str{G}$ contains a homomorphic image of a non-$\Monoid$-metric cycle if and only if it contains a non-$\Monoid$-metric cycle as an induced substructure (i.e.~a monomorphic image).
\item \label{cor:suaer1:4} $\mathcal M_\Monoid$ is an amalgamation class closed for strong amalgamation.
\end{enumerate}
\end{proposition}
\begin{example}
\begin{figure}
\centering
\includegraphics{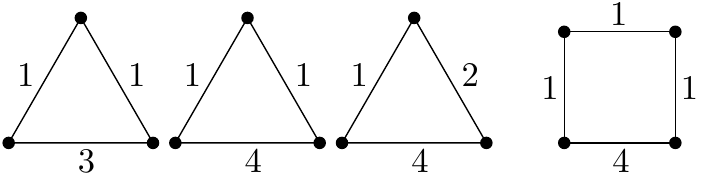}
\caption{All non-metric cycles with distances $1,2,3$ and $4$.}
\label{fig:4graphs}
\end{figure}
Consider the monoid $\Monoid_S$ as in Example~\ref{example1}~(\ref{ex:Smetric}) for $S=\{1,\allowbreak 2,\allowbreak 3,\allowbreak 4\}$.
Figure~\ref{fig:4graphs} depicts all non-metric cycles, given by Proposition~\ref{prop:cycles} (\ref{cor:suaer1:3}), which prevent $\Monoid$-graphs from having a completion to an $\Monoid$-metric space.
\end{example}
\begin{proof}
All three statements are consequences of associativity of $\oplus$.

\medskip
\paragraph{(\ref{cor:suaer1:1})}
Assume that $\str{G}$ is $\Monoid$-metric. First we show that the completion
described will give an $\Monoid$-metric space by verifying that $d'$ satisfies the triangle inequality. Take any three vertices $u,v,w$. Combine the walks
$\mathcal W(u,w)$ and $\mathcal W(w,v)$ to get a walk from $u$ to $v$ in $\str{G}$ whose length is $d'(u,w)\oplus d'(w,v)$.
It follows that $d'(u,v)\mleq d'(u,w)\oplus d'(w,v)$.

We have shown that $d'$ forms an $\Monoid$-metric space on vertices of $\str{G}$ but it still needs to be checked that $d_\str{G}(u,v)=d'(u,v)$ whenever $d_\str{G}(u,v)$ is defined.
 We show a stronger claim: if $\str{B}$ is a completion of $\str{G}$ to an $\Monoid$-metric space
then $d_\str{B}(u,v)\mleq d'(u,v)$ for every $u\neq v\in \str{G}$.

Suppose, for a contradiction, that there are vertices $u\neq v\in \str G$ such that $d_\str{B}(u,v)\mgt d'(u,v)$. By definition of $d'$, there is a path $\mathcal W(u,v)$ in $\str{G}$ with $d'(u,v)$ being its length and then $d_\str{B}(u,v)\mgt d'(u,v)$ contradicts $\str{B}$ being an $\Monoid$-metric space.

\medskip
We first show (\ref{cor:suaer1:2}) and (\ref{cor:suaer1:3}) for connected $\Monoid$-graphs only.
\medskip
\paragraph{(\ref{cor:suaer1:2})}
Assume that $\str{G} = (G,d)$ is non-$\Monoid$-metric and let $\str A = (G, d')$ be its shortest path completion.
$\str{G}$ being non-$\Monoid$-metric means that there are vertices $u\neq v\in G$ with $d'(u,v) \mlt d(u,v)$. But that means
that the $\Monoid$-length of $\mathcal W(u,v)$ is strictly less than $d(u,v)$, hence this path together
with the edge $u,v$ forms a non-$\Monoid$-metric cycle.

\medskip
\paragraph{(\ref{cor:suaer1:3})} Again first consider connected $\Monoid$-graphs only.
One implication is trivial. The other follows from $\oplus$ being monotonous with respect to $\mleq$ -- it is enough to take the minimal subcycle of the homomorphic image of the non-$\Monoid$-metric cycle containing the edge $v_1v_n$ (see (\ref{cor:suaer1:2})).

\medskip
Now it is easy to see that both (\ref{cor:suaer1:2}) and (\ref{cor:suaer1:3}) also hold for $\Monoid$-graphs that are not connected, because every such $\str{G}$ can be turned into connected one 
by adding new edges connecting individual components without introducing new cycles.

\medskip
\paragraph{(\ref{cor:suaer1:4})}
Given $\str{A},\str{B}_1,\str{B}_2\in \mathcal M_\Monoid$ it is easy to see that
the free amalgamation of $\str{B}_1$ and $\str{B}_2$ over $\str{A}$ contains no
embedding of any non-metric cycle as described in the previous paragraph.
\end{proof}

\section {Archimedean monoids}
\label{sec:Mmetricarch}
In this section we use the machinery introduced in Section~\ref{sec:multiamalgamation} to show that for an archimedean monoid $\Monoid$ the class $\vv{\mathcal M}_\Monoid$ of all linearly ordered $\Monoid$-metric spaces is Ramsey.\footnote{In full generality, $\vv{\mathcal M}_\Monoid$ will only contain structures where the order is \emph{convex} (see Definition~\ref{defn:Js}). For archimedean monoids, these two notions coincide.} This is supposed to serve as a warm-up for Section~\ref{sec:Metric}, where we deal with general distance monoids and the means are considerably more difficult.

\begin{lemma}
\label{lem:archimedean}
Let $\EMonoid$ be an archimedean distance monoid. Then for every $a,b\in M$, $b\succ 0$ either $a\oplus b\succ a$ or $a$ is the maximal element of $\Monoid$.
\end{lemma}
\begin{proof}
Assume the contrary and consider $a,b$ such that $a\oplus b=a$ and $a$ is not the maximal element of $\Monoid$.
 In this case also $a\oplus (n\times b)=a$ for every $n$.  Because $a$ is not maximal element there exists $c\in \Monoid$ such that $a\mlt c$.
 Because $\Monoid$ is archimedean
we however know that there is $n$ such that $n\times b\succeq c\succ a$. A contradiction with monotonicity of $\oplus$.
\end{proof}

The following lemma is the basic tool used to show local finiteness condition needed by
Theorem~\ref{thm:mainstrongclosures}.

\begin{lemma}
\label{lem:nonSmetric}
Let $\EMonoid$ be an archimedean distance monoid and let $S\subseteq M$ be finite. Then there exists $n=n(S)$ such that for every non-$\Monoid$-metric cycle $\str{C}$ such that all distances in $\str{C}$ are from $S$ it holds that $\str{C}$ has at most $n$ vertices.
\end{lemma}
\begin{proof}
Because $\Monoid$ is archimedean, for every $a,b\in \Monoid$ there exists a smallest $m=m(a,b)$ such that $a\mleq m\times b$. Let $n = \max\left\{m(a,b); a,b\in S\right\}$. Then by Lemma~\ref{lem:archimedean} it follows for every $a_1, a_2, \ldots, a_n \in S$ that $a_1\oplus a_2 \oplus \cdots \oplus a_n$ is at least the largest element in $S$ and hence if $\str{C}$ has at least $n+1$ vertices, it cannot be non-$\Monoid$-metric.
\end{proof}

Now we are ready to show that the class of all finite $\Monoid$-metric spaces
has Ramsey expansion.

\begin{theorem}
\label{thm:Smetricarch}
Let $\EMonoid$ be an archimedean distance monoid.
Then the class of all finite $\Monoid$-metric spaces with free, i.e.~arbitrary, ordering of vertices, $\vv{\mathcal M}_\Monoid$, is a
Ramsey class.
\end{theorem}

\begin{proof}
We will show that $\vv{\mathcal M}_\Monoid$ is an $\mathcal R$-multiamalgamation class, where $\mathcal R$ is the class of all finite linearly ordered complete $\Monoid$-graphs.

Let $\str{C}_0$ be an arbitrary finite linearly ordered $\Monoid$-graph. We will show that there exists an $n=n(\str{C}_0)$ satisfying the following:

Let $\str{C}$ be a $\Monoid$-graph with an additional binary relation $\leq_\str{C}$ which has a homomorphism-embedding to $\str{C}_0$. Further suppose that every substructure of $\str{C}$ on at most $n$ vertices is $\Monoid$-metric. Then $\str{C}$ is $\Monoid$-metric.

If we show that, we thereby check the conditions of Theorem~\ref{thm:mainstrongclosures} (the order can be completed arbitrarily, the strong amalgamation property is given by Proposition~\ref{prop:cycles} (\ref{cor:suaer1:4}) and remaining assumptions are trivial).

Let $S$ be the set of distances which appear in $\str{C}_0$. As $\str{C}_0$ is finite, $S$ is clearly finite, too. Take $n=n(S)$ from Lemma~\ref{lem:nonSmetric}. Then every non-metric cycle has at most $n$ vertices. And as every non-$\Monoid$-metric graph contains a non-$\Monoid$-metric cycle by Proposition~\ref{prop:cycles} (\ref{cor:suaer1:3}), the statement follows.
\end{proof}

\section {General distance monoids}
\label{sec:Metric}
In this section we generalise the construction to distance monoids in full
generality.  In particular, unlike~\cite{Conant2015} we do not need the notion
of semi-archimedean monoids. (A distance monoid $\EMonoid$ is {\em
semi-archimedean} if, for all $r, s \in M$, $r,s\mgt 0$, if $n\times r\mlt s$
for all $n>0$ then $r\oplus s=s$.)

The main difficulties in generalizing Theorem~\ref{thm:Smetricarch} come from
the fact that there is no direct equivalent to Lemma~\ref{lem:nonSmetric}.
Consider, for example, spaces with distances $1$ and $3$. Every such metric
space consists of a disjoint union of balls of diameter 1 separated by distance
3. Every cycle having one distance $3$ and rest of distances $1$ is
forbidden regardless of the number of its vertices.

To overcome this problem we need to precisely characterise definable equivalences
in $\Monoid$-metric spaces (i.e.~formulas $\varphi(x,y)$ such that in the \Fraisse{} limit $\mathcal U$
the relation $x\sim y\Leftrightarrow \mathcal U\vDash\varphi(x,y)$ is an equivalence relation)
and represent them by means of artificial vertices
and functions. In the language of model theory, we are going to eliminate
imaginaries, see~\cite{Hubicka2016} for details.

\subsection {Blocks and block equivalences}
As we will show, the definable equivalences are related to archimedean submonoids of
$\Monoid$-metric spaces. 
The following is a generalization of a definition by Sauer~\cite{Sauer2013}.
\begin{definition}
Given a distance monoid $\EMonoid$, a {\em block} $\Block$ of $\Monoid$ is a subset of $M$ such that either
\begin{enumerate}
\item $\Block=\{0\}$, or
\item $0\notin \Block$ and $\{0\}\cup \Block$ induces a maximal archimedean submonoid of $\Monoid$.
\end{enumerate}
Given a block $\Block$ we will denote by $\Monoid_\Block$ the archimedean submonoid induced by it.
\end{definition}
The basic properties of blocks can be summarized as follows.
\begin{lemma}
\label{lem:blocks}
Given a distance monoid $\EMonoid$ it holds that:
\begin{enumerate}
\item\label{lem:blocks:unique} For every $a\in M$ there exists a unique block $\Block_a$ containing $a$. 
\item\label{lem:blocks:sameblock} Let $a,b\in \Monoid$. If there exist $m, n$ such that $m\times a \mgeq b$ and $n\times b \mgeq a$, then $a,b$ are in the same block.
\end{enumerate}
\end{lemma}
\begin{proof}
(\ref{lem:blocks:unique}) Let $$\Block_a = \left\{b \in \Monoid;  (\exists n)(n\times a \mgeq b) \land (\exists n)(n\times b \mgeq a)\right\}.$$ It is easy to check that $\Monoid_{\Block_a}=(\Block_a\cup \{0\},\oplus,0,\mleq)$ is an archimedean submonoid of $\Monoid$. The maximality and uniqueness follows from the fact that no $b\in M\setminus (\Block_a\cup \{0\})$ can be in the same archimedean submonoid as $a$.

(\ref{lem:blocks:sameblock}) follows from the proof of (\ref{lem:blocks:unique}).
\end{proof}
Note that the relation $$R\subset M^2 = \left\{(a,b)\in M^2;  (\exists n)(n\times a \mgeq b) \land (\exists n)(n\times b \mgeq a)\right\}$$ used in the proof is an equivalence relation on $M$ whose equivalence classes are precisely the blocks.

Given a distance monoid $\EMonoid$ and $a\in M$, we will always denote by $\Block_a$ the unique block of $\Monoid$ containing $a$ given by Lemma~\ref{lem:blocks}~(\ref{lem:blocks:unique}).

\begin{lemma}\label{lem:blockorder}
Let $\EMonoid$ be a distance monoid and $a,b,c\in M$ such that $a,c\in\Block$ and $a\mlt b \mlt c$, then $b\in \Block$.
\end{lemma}
\begin{proof}
Take any $a,c\in\Block$ and $b\in\Block'$. As $a,c\in\Block$, there is $n$ such that $n\times a\mgeq c$. But then also $n\times a\mgeq b$ and hence by Lemma~\ref{lem:blocks}~(\ref{lem:blocks:sameblock}) $a$ and $b$ are in the same block.
\end{proof}
This means that in the order $\mleq$ blocks form intervals and hence $\mleq$ induces a linear order of blocks of $\Monoid$. We will denote this order by the same symbol $\mleq$ (namely we say $\Block \mleq \Block'$ if for every $a\in \Block, b\in\Block'$ it holds that $a\mleq b$).

\begin{definition}
Let $\str{A}$ be an $\Monoid$-metric space and $\Block$ block of $\Monoid$. 

\begin{enumerate}
\item A \emph{block equivalence $\sim_{\Block}$}
on vertices of $\str{A}$ is given by $u\sim_{\Block} v$ whenever there exists $a\in \Block$ such that $d(u,v)\mleq a$.

\item A {\em ball of diameter $\Block$} in $\str{A}$ is any equivalence class of $\sim_\Block$ in $\str{A}$.
\end{enumerate}
\end{definition}
Note that a block $\Block$ does not need to contain maximal element.

To verify that for every block $\Block$ the relation $\sim_{\Block}$ is indeed an equivalence relation it suffices to check transitivity.  Given a triangle with distances $a,b,c$, if there exists $a'\in \Block$ such that $a\preceq a'$ and $b'\in \Block$ such that $b\preceq b'$ it also holds that $c\preceq a\oplus b\preceq a'\oplus b'\in \Block$.

\medskip

Note that, in an $\Monoid$-metric space, there can be many types of pairs of balls of the same diameter. For example, consider monoid $\Monoid$ given by Example~\ref{example1} (\ref{ex:Smetric}) for $S=\{1,3,5\}$. $\Monoid$ has three blocks: $\{0\}$, $\{1\}$ and $\{3,5\}$. If there are vertices $u,v$ such that $d(u,v)=3$ then in fact for every pair of vertices $u',v'$ such that $u\sim_{\{1\}} u'$ and $v\sim_{\{1\}} v'$ it holds that $d(u',v')=3$. In other words, there are two types of pairs balls of diameter $1$ in $\Monoid$-metric spaces, those in distance 3 and those in distance 5.

To formalize this, for block $\Block$ and a distance $\ell\in \Monoid$ such that $\Block_\ell \mgt \Block$ we denote by $t(\Block,\ell)$ the set of all distances which can appear between two balls of diameter $\Block$ provided that $\ell$ appears there. (So, in the previous example, we have $t(\{1\}, 3) = \{3\}$ and $t(\{1\}, 5) = \{5\}$.) We will call the sets $t(\Block,\ell)$ \emph{block-types}.

\begin{observation}\label{obs:blocktypes}
The following holds about $t(\Block,\ell)$:
\begin{enumerate}
\item $\ell\in t(\Block,\ell)$,
\item $t(\Block,\ell)\subseteq \Block_\ell$,
\item either $t(\Block,\ell) = t(\Block,\ell')$ or $t(\Block,\ell) \cap t(\Block,\ell') = \emptyset$, and
\item if $\Block'\mgeq \Block$ then $t(\Block',\ell)\supseteq t(\Block, \ell)$.
\end{enumerate}
\end{observation}

\subsection {Important and unimportant summands}\label{subsec:nonimportant}
This rather technical part is the key to obtaining a locally finite
description of $\mathcal{M}_\Monoid$ (needed for Theorem~\ref{thm:mainstrongclosures}).

Given $\EMonoid$ and $S\subseteq M$, we will denote by $S^\oplus$ the set of all values which can be
obtained as nonempty sums of values in $S$. (Thus $S^\oplus\cup \{0\}$ forms the submonoid of $\Monoid$ generated by $S$.)

Blocks of a monoid may be infinite and may not contain a maximal element which
would be useful in our arguments (those maximal elements are referred to as jump numbers in~\cite{Hubicka2016}).  For a fixed finite $S\subseteq M$ we seek
a sufficient approximation $\mus(\Block,S)$ of the maximal element of $\Block$ (which might not exist) given by the
following lemma.  The name $\mus$ means \emph{maximum useful} distance (with
respect to $\Block$ and $S$).
\begin{lemma}\label{lem:obstaclemus}
Let $\EMonoid$ be a distance monoid with finitely many blocks and $S\subseteq M$ be a finite subset of $M$. 
 Then for every nonzero block $\Block$ of $\Monoid$ there is a distance $\mus(\Block, S) \in \Block$ such that for every $\ell\in S$ and $e\in S^\oplus$ one of the following holds:
\begin{enumerate}
\item $e\oplus \mus(\Block, S) \mgeq \ell$, or
\item $e\oplus b\mlt \ell$ for every $b \in \Block$ (and thus also for every $b\in\Block'$, where $\Block'\mleq \Block$).
\end{enumerate}
\end{lemma}
Note that $\mus(\Block, S)$ is not necessarily unique. In particular, if a distance $x$ satisfies the conditions on $\mus(\Block, S)$, then so do all the distances in $\Block$ larger than $x$.
\begin{example}
It is always possible to put $\mus(\Block, S)$ to be the maximal element of $\Block$ if it exists. The choice may be more difficult for blocks with no maximal element.
To clarify this consider the monoid with infinitesimals given in Example~\ref{example1} (\ref{infini}). This monoid has three blocks.  $\Block_0=\{0\}$, $\Block_1$ consists of infinitesimals and $\Block_2$ of all remaining values. For $S=\{\mathrm{dx}, 1, 2+3\mathrm{dx}\}$ we can put  $\mus(\Block_1,S)=3\mathrm{dx}$ and $\mus(\Block_2,S)=1+3\mathrm{dx}$. Note that $\mus(\Block_1,S)=3\mathrm{dx} \notin S$.
\end{example}
\begin{proof}[Proof of Lemma~\ref{lem:obstaclemus}]
Let $S$ be a fixed finite subset of $M$.
Enumerate nonzero blocks of $\Monoid$ as $\Block_1\mgeq \Block_2 \mgeq \cdots \mgeq \Block_p$.

Given a block $\Block$ and distances $\ell, e \in M$, define $f(\Block, \ell, e)$ to be some (for example the smallest, if it exists) $a\in \Block$ such that $\ell\mleq e\oplus a$ or zero if $\ell \mgt e\oplus a$ for all $a\in\Block$. Further, given a block $\Block$ and a distance $e\in \Block$, define:
$$X(\Block, e) = \left\{ a\in \Block\cup\{0\}; a\mleq e \hbox{ and } \exists_{b_1,  \ldots, b_m \in \Block\cap S}\thinspace a=b_1\oplus\cdots\oplus b_m \right\}.$$
Here an empty sum equals to zero. Hence $0\in X(\Block, e)$ for every choice of $\Block$ and $e$.

Observe that $X(\Block, e)$ is finite for any choice of $\Block$ and $e$:
As $\Monoid_\Block$ is archimedean, we have for each $b\in \Block\cap S$ some finite $n(b)$ such that $n(b)\times b\mgeq e$. As $\Block\cap S$ is also finite and $\mleq$ is a congruence for $\oplus$, $X(\Block, e)$ is finite.

Now, for a given $\ell\in S$, we define by induction on $i$ the finite sets $X_i(\ell)$, $1\leq i\leq p$, and distances $d_i(\ell)$:
\begin{align*}
d_i(\ell) &= 
  \begin{cases} 
   \text{an arbitrary }a\in \Block_i  & \text{if } \Block_i\mgt \Block_\ell, \\
   \ell & \text{if } \Block_i = \Block_\ell, \\
   \max_{e\in X_{i-1}(\ell)} f(\Block_i, \ell, e) & \text{if } \Block_i\mlt \Block_\ell;
  \end{cases}\\
X_i(\ell) &= 
  \begin{cases} 
   \emptyset & \text{if } \Block_i\mgt \Block_\ell, \\
   X(\Block_\ell, d_\ell(\ell)) & \text{if }\Block_i = \Block_\ell, \\
   X_{i-1}(\ell) \oplus X(\Block_i, d_i(\ell)) & \text{if } \Block_i\mlt \Block_\ell;
  \end{cases}
\end{align*}
where $A\oplus B = \{a\oplus b; a\in A, b\in B\}$.

Note that $X_{i-1}(\ell) \subseteq X_i(\ell)$.

\begin{claim}\label{claim:obstaclemus:2}
The following two statements are true for every $1\leq i \leq p$ and $\ell\in S$:
\begin{enumerate}
\item \label{claim:obstaclemus:2:2} For every $e\in S^\oplus$ either $e\oplus d_i(\ell) \mgeq \ell$, or $e\oplus b\mlt \ell$ for every $b \in \Block_j$, where $j\geq i$.
\item \label{claim:obstaclemus:2:1} Let $e\in \left(S \cap \bigcup_{j\leq i} \Block_j\right)^\oplus$ be a distance such that $e\mlt \ell$ and there is $b\in\Block_j$, $j > i$ with $e\oplus b \mgeq \ell$. Then $e\in X_{i}(\ell)$.
\end{enumerate}
\end{claim}
By (\ref{claim:obstaclemus:2:2}) we obtain the statement of Lemma~\ref{lem:obstaclemus}:
Put 
$$\mus(\Block_i, S) = \max_{\ell\in S} d_i(\ell)$$
and as a special case, if for some $i$ we would have $\mus(\Block_i, S) = 0$ choose $\mus(\Block_i, S)\in \Block_i$ arbitrarily.

\medskip
Thus it remains to prove Claim~\ref{claim:obstaclemus:2}.
We will proceed by induction on $i$ with $\ell$ fixed. If $\Block_i\mgeq \Block_\ell$, both statements of Claim~\ref{claim:obstaclemus:2} are trivially satisfied. This also proves the base case, so now we can suppose that $\Block_i \mlt \Block_\ell$.

First suppose that both statements are true up to some $i-1$, but statement~(\ref{claim:obstaclemus:2:2}) is false for $i$, i.e.~there are distances $e_1, e_2, \ldots, e_k \in S$, $b\in \Block_j$, $j\geq i$ and $e = e_1\oplus e_2\oplus \cdots \oplus e_k\in S^\oplus$ with $e\oplus d_i(\ell) \mlt \ell$, but $e\oplus b\mgeq \ell$. Without loss of generality we may assume $j = i$, i.e.~$b\in \Block_i$. Let $(e_i')$ be the subsequence of $(e_i)$ containing only distances from blocks larger than $\Block_i$ (i.e.~blocks $\Block_j$ for $j<i$) and $(e_i'')$ be the complement of $(e_i')$. Let $e' = \bigoplus_i e_i'$ and $e'' = \bigoplus_i e_i''$. Clearly $e = e' \oplus e''$. Finally let $b' = b \oplus e'' \in \Block_i$.

Now $e'\oplus d_i(\ell) \mlt \ell$ and $e'\oplus b'\mgeq \ell$. But $e'$ can be expressed as a sum of elements from $S \cap \left(\bigcup_{j<i} \Block_j\right)$, hence $e'\in X_{i-1}(\ell)$ by the induction hypothesis for~(\ref{claim:obstaclemus:2:1}) and this is a contradiction with the definition of $d_i(\ell)$.

To prove~(\ref{claim:obstaclemus:2:1}), assume that it holds up to $i-1$ and that~(\ref{claim:obstaclemus:2:2}) holds up to $i$. Take an arbitrary $e\in\left(S \cap \bigcup_{j\leq i} \Block_j\right)^\oplus$ with $e\mlt \ell$ such that there is $b\in \Block_{j}$ with $e\oplus b \mgeq \ell$, where $j > i$. We can again assume that $b\in \Block_{i+1}$. We want to prove that $e\in X_i(\ell)$.

As $e\in\left(S \cap \bigcup_{j\leq i} \Block_j\right)^\oplus$, this, by definition, means, that there is a finite sequence $(e_j)$ of elements of $S \cap \bigcup_{j\leq i} \Block_j$ such that $e = \bigoplus_j e_j$.

As in the previous point, let $(e_j')$ be the subsequence of $(e_j)$ containing only distances from blocks larger than $\Block_i$ (i.e.~blocks $\Block_j$ for $j<i$) and $(e_j'')$ be the complement of $(e_j')$. Let $e' = \bigoplus_j e_j'$ and $e'' = \bigoplus_j e_j''$. Clearly $e = e' \oplus e''$, $e'\mlt \ell$ and $e''\in \Block_i$ (using the induction hypothesis we may assume that $e''\neq 0$).

As in the proof of~(\ref{claim:obstaclemus:2:2}), we can observe that $e'\in X_{i-1}$.

If $e''\mgeq d_i(\ell)$, then already $e$ would be at least $\ell$, which contradicts the assumptions. Hence $e''\mlt d_i(\ell)$ and then from the definition of $X(\Block, a)$ it follows that $e'' \in X(\Block_i, d_i(\ell))$, and thus $e\in X_i(\ell)$. This concludes the proof of Claim~\ref{claim:obstaclemus:2}.
\end{proof}

The following proposition (an easy consequence of Lemma~\ref{lem:obstaclemus}) is the main result of this section which will be used in proving the local finiteness property:
\begin{proposition}
\label{prop:Sequivalence}
Let $\EMonoid$ be a distance monoid with finitely many blocks and $S\subset M$ be a finite subset of $M$. There exists $n=n(S)$ such that for every $\ell\in S$ and every sequence $e_1, e_2, \ldots, e_k \in S$ with $\ell \mgt e_1\oplus e_2 \oplus \cdots \oplus e_k$ there is a sequence $f_1, f_2, \ldots, f_m\in S$ satisfying the following properties:
\begin{enumerate}
\item $(f_i)$ is a subsequence of $(e_i)$;
\item $m < n$; and
\item if $(f_i)$ is a proper subsequence of $(e_i)$, $a$ is the largest $e_i$ not in the sequence $(f_i)$, and $b\in\Block_a$  an arbitrary distance, then $\ell\mgt b\oplus f_1\oplus f_2 \oplus \cdots \oplus f_m$. (Here $\Block_a$ is the block containing $a$.)
\end{enumerate}
\end{proposition}
We will call the distances $(f_i)$ \emph{important}.
\begin{proof}
Let $\Block_1, \Block_2, \ldots, \Block_b$ be blocks of $\Monoid$ which intersect $S$ nontrivially (for each $\Block_i$ there is an $a_i\in S$ with $a_i\in \Block_i$). For each $\Block_i$ define $m_i$ to be the minimum element of $\Block_i\cap S$ and let $n_i$ be the smallest integer such that $n_i\times m_i \mgeq \mus(\Block_i, S)$ (because $\Block_i$ is archimedean such $n_i$ exists). Put $$n = n(S) = 1+\sum_i n_i.$$

Let $\ell, e_1, e_2, \ldots, e_k \in S$ be given with $\ell \mgt e_1\oplus\cdots\oplus e_k$. 
Now we shall construct the sequence $(f_i)$ using the following algorithm:
For each $\Block_i$ create a variable $c_i$ which is initially set to zero. 
For each $i$ from $1$ to $k$ do the following:
\begin{enumerate}
\item Let $\Block_j$ be the block containing $e_i$.
\item If $c_j \mleq \mus(\Block_j, S)$, put $e_i$ into the sequence and increment $c_j \leftarrow c_j\oplus e_i$.
\end{enumerate}

One can easily check that $(f_i)$ satisfies all properties from the statement.
\end{proof}

\subsection {Convex ordering of $\mathcal M_\Monoid$}
To obtain a Ramsey class we need to define a notion of ordering for classes $\mathcal M_\Monoid$. The following definition is a generalization of convex ordering for equivalences and metric spaces (see e.g.~\cite{The2010}).
Recall that we interpret an $\Monoid$-metric space as a relational structure $\str{A}$ in the language $L_\Monoid$ with (possibly infinitely many) binary relations $\rel{}{s}$, $s\in M^{\succ 0}$.

\begin{definition}
\label{defn:Js}
Given a distance monoid $\EMonoid$ we expand the language $L_\Monoid$ to the language $L^+_\Monoid$ by a binary relational symbol $\leq$ representing a linear order.

Given an $\Monoid$-metric space $$\str{A}=(A,(\rel{A}{s})_{s\in M^{\succ 0}}),$$ we say that an $L^+_\Monoid$-struc\-ture $$\str{A}^+=(A^+,(\rel{A^+}{s})_{s\in M^{\succ 0}},\leq_{\str{A}^+})$$ is a {\em convexly ordered expansion} of $\str{A}$ if
\begin{enumerate}
\item $A = A^+$, $\rel{A}{s} = \rel{A^+}{s}$ for all $s\in M^{\succ 0}$ and
\item $\leq_{\str{A}^+}$ is a linear ordering of $A^+$ such that for every block $\Block$ of $\Monoid$ and every $a, b, c\in A^+$ with $a \sim_\Block b$ and $a\nsim_\Block c$ it holds that $a\leq_{\str{A}^+} c$ if and only if $b\leq_{\str{A}^+} c$ (that is, every ball forms an interval in $\leq_{\str{A}^+}$).
\end{enumerate}
We will denote by $\vv{\mathcal M}_\Monoid$ the class of all convexly ordered $\Monoid$-metric spaces.
\end{definition}

We will now consider a further (more technical) expansion of the class $\mathcal M_\Monoid$
which will make it possible to apply Theorem~\ref{thm:mainstrongclosures}. One-to-one correspondence between structures in both expansions will let us show the Ramsey property for $\vv{\mathcal M}_\Monoid$.  For this let $B_\Monoid$ be the set of all non-zero non-maximal (in the block order) blocks of $\Monoid$.
\begin{definition}\label{defn:lstar}
Given a distance monoid $\EMonoid$ with finitely many blocks denote by $L^\star_\Monoid$ the expansion of 
the language $L^+_\Monoid$ adding
\begin{enumerate}
\item unary function symbols $\func{}{\Block}$ for every block $\Block\in B_\Monoid$,
\item unary function symbols $\func{}{\Block,\Block'}$ for every pair of blocks $\Block, \Block'\in B_\Monoid$ such that $\Block\mlt \Block'$, and
\item binary symmetric relation symbols $\rel{}{t}$ for every $t=t(\Block,\ell)$ where $\Block\in B_\Monoid$ and $\Block_\ell \mgt \Block$.
\end{enumerate}
For a given convexly ordered metric space $\str{A}\in \vv{\mathcal M}_\Monoid$, denote by $L^\star(\str{A})$ the {\em $L^\star$-expansion} (or {\em lift}) of $\str{A}$ created by the following procedure:
\begin{enumerate}
\item For every $\Block\in B_\Monoid$ enumerate balls of diameter $\Block$ in $\str{A}$ as $E^1_\Block, E^2_\Block,\ldots,\allowbreak E^{n_\Block}_\Block$ in the order of $\leq_\str{A}$ (recall that balls are linear intervals in $\leq_\str{A}$ and thus this is well defined).

\item For every $\Block\in B_\Monoid$ and $1\leq i\leq n_\Block$ add a new vertex $v^i_\Block$.
\item For every $\Block\in B_\Monoid$,  $1\leq i\leq n_\Block$ and $v\in E^i_\Block$ put $\nbfunc{L^\star(\str{A})}{\Block}(v)=v^i_\Block$.
\item For every pair of blocks $\Block, \Block'\in B_\Monoid$ such that $\Block\mlt \Block'$ and every $1\leq i\leq n_\Block$ put $\nbfunc{L^\star(\str{A})}{\Block,\Block'}(v^i_\Block)=\nbfunc{L^\star(\str{A})}{\Block'}(v)$  where $v$ is some vertex of $E^i_\Block$.
\item For every $\Block\in B_\Monoid$ and $1\leq i < j\leq n_\Block$ put $(v^i_\Block,v^j_\Block),(v^j_\Block,v^i_\Block)\in\nbrel{L^\star(\str{A})}{t(\Block, \ell)}$, where $\ell = d(u,v)$ for arbitrary $u\in E^i_\Block$ and $v\in E^j_\Block$ (by Observation~\ref{obs:blocktypes} this does not depend on the choice of $u$ and $v$).
\item Extend linear ordering $\leq_\str{A}$ to ordering $\leq_{L^\star(\str{A})}$ by putting
\begin{enumerate}
\item $v\leq_{L^\star(\str{A})} v' $ for every $v\in A$, $v'\not \in A$,
\item $v^i_\Block \leq_{L^\star(\str{A})} v^j_\Block$ for every $\Block\in B_\Monoid$ and $1\leq i\leq j\leq n_\Block$, and
\item $v^i_\Block \leq_{L^\star(\str{A})} v^j_{\Block'}$ for every $\Block\prec \Block'$, $1\leq i\leq n_\Block$ and $1\leq j\leq n_{\Block'}$.
\end{enumerate}
\end{enumerate}
%In  $L^\star(\str{A})$ we will call vertices of $\str{A}$ the {\em original vertices} and the added $v^i_\Block$ vertices the {\em ball vertices}.
We will denote by $\mathcal M^\star_\Monoid$ the class of all substructures of all $L^\star(\str{A})$, $\str{A}\in \vv{\mathcal M}_\Monoid$.
\end{definition}

Observe that for every block $\Block$ of a distance monoid $\Monoid$, every structure $\str{A} \in \mathcal M^\star_\Monoid$ and every two vertices $u,v\in\str{A}$ we have $\nbfunc{L^\star(\str{A})}{\Block}(u) = \nbfunc{L^\star(\str{A})}{\Block}(v)$ if and only if $u\sim_\Block v$. Thus, in the model-theoretical language, $L^\star(\str{B})$ is essentially $\str{B}^{\mathrm{eq}}$ for $\str{B}\in \vv{\mathcal M}_\Monoid$.

\subsection{Proof of the main result}
\label{subsec:main}

We first prove the main result assuming that there are only finitely many blocks and later generalize it to the infinite case. To improve the presentation, we shall split it into several auxiliary lemmas.

Let $\EMonoid$ be a distance monoid with finitely many blocks and let $\str{A}$ be an $L^\star_\Monoid$-structure such that every irreducible substructure of $\str A$ belongs to $\mathcal M^\star_\Monoid$. We say that a vertex $v\in \str{A}$ is an \emph{original vertex}, if the functions $\func{A}{\Block}$ are defined on this vertex (i.e.~it is contained in the underlying metric space) for every block $\Block$. Otherwise we call the vertex a \emph{ball vertex}.

\begin{remark}
We needed to let $\mathcal M^\star_\Monoid$ be the class of all substructures of some $L^\star(\str{A})$, $\str{A}\in \vv{\mathcal M}_\Monoid$ in order to ensure that it is hereditary (which is required by Theorem~\ref{thm:mainstrongclosures}). In practice, this amounts to allowing our structures to contain some ball vertices to which no original vertices point.
\end{remark}

%Note that every ball vertex is an image of some function $\nbfunc{\str{C}}{\Block}$ (follows from the fact that every vertex of $\str{C}$ is in some copy of $\str{B}$).

In order to use Theorem~\ref{thm:mainstrongclosures} we need to prove that $\mathcal M^\star_\Monoid$ is a strong amalgamation class and that it is a locally finite subclass of some Ramsey class $\mathcal R$ (in our case $\mathcal R$ will be the class of all finite $L^\star_\Monoid$-structures which is Ramsey by Theorem~\ref{thm:NRclosures}). Both these statements have very similar flavours and we will obtain them as corollaries of something slightly stronger.

Let $\str A$ be an $L^\star_\Monoid$-structure and let $D'\subseteq A$ be a subset of its vertices. Then we say that an $L^\star_\Monoid$-structure $\str D$ is a \emph{closure} of $D'$ in $\str A$ if all the following conditions hold:
\begin{enumerate}
\item $\str D$ is a substructure of $\str A$ (and thus if $\nbfunc{}{}$ is a function symbol of $L^\star_\Monoid$, $v\in D$ and there is $u\in A$ such that $\func{A}{}(v) = u$, then also $u\in D$);
\item $D'\subseteq D$;
\item $\str D$ is minimal such substructure.
\end{enumerate}

\begin{lemma}\label{lem:finite4value2:1}
Let $\EMonoid$ be a distance monoid with finitely many blocks and let $\str A$ be an $L^\star_\Monoid$-structure such that every irreducible substructure of $\str A$ belongs to $\mathcal M^\star_\Monoid$. Then for every pair of original vertices $u,v\in \str{A}$ and every block $\Block\in B_\Monoid$ such that there exists a walk from $u$ to $v$ where every distance is
at most some $a\in \Block$ we have $\func{A}{\Block}(u)=\func{A}{\Block}(v)$.
\end{lemma}
\begin{proof}
Let $u'$ and $w'$ be two neighbouring vertices in the walk.  Because 
the closure every edge of $\str{A}$ is irreducible, it follows that $\nbfunc{\str{A}}{\Block}(u')=\nbfunc{\str{A}}{\Block}(w')$.
Consequently, all vertices of the walk have same value of $\nbfunc{\str{A}}{\Block}$.
\end{proof}

For an $L^\star_\Monoid$-structure $\str A$ we will denote by $\str A^-$ the same structure without the order.
\begin{lemma}\label{lem:ordercompletion}
Let $\EMonoid$ be a distance monoid with finitely many blocks and let $\str A$ be an $L^\star_\Monoid$-structure such that every irreducible substructure of $\str A$ belongs to $\mathcal M^\star_\Monoid$ and moreover $\leq_\str A$ has a linear extension. If there is $\str B\in \mathcal M^\star_\Monoid$ such that $\str B^-$ is a strong completion of $\str A^-$ then there exists $\str B'\in \mathcal M^\star_\Monoid$ which is a strong completion of $\str A$
\end{lemma}
\begin{proof}
We can assume that $A\subseteq B$ and that the inclusion is a homomorphism-embedding $\str A^-\to\str B^-$. It follows that there is a linear order $\leq_0$ of vertices of $\str B^-$ such that it extends $\leq_\str A$. We will now define a linear order $\leq_1$ such that $\str B^-$ together with $\leq_1$ will give the desired $\str B'$.

We explicitly define the order $\leq_1$ as follows:
\begin{enumerate}
	\item Enumerate blocks of $\Monoid$ as $\Block_1\mgt \Block_2\mgt\cdots\mgt \Block_m$.
	\item For every pair $u,v$ of original vertices of $\str{B}^-$ put $u\leq_1 v$ if the sequence of vertices $(\nbfunc{\str B^-}{\Block_i}(u))_{i=1}^m$ is in order $\leq_0$ lexicographically before $(\nbfunc{\str B^-}{\Block_i}(v))_{i=1}^m$ or they are the same and $u\leq_0 v$.\footnote{If $a=(a_1, \ldots, a_p)$ and $b=(b_1, \ldots, b_p)$ are sequences then $a$ is lexicographically before $b$ if there is $1\leq i\leq p$ such that $a_i \mlt b_i$ and $a_j = b_j$ for all $j<i$.}

	\item For every pair $u,v$ of ball vertices such that $u$ corresponds to block $\Block_i$ and $v$ to block $\Block_j$ put $u\leq_1 v$ if one of the following holds:
\begin{enumerate}
\item $i<j$,
\item $i=j$ and the sequence $(\nbfunc{\str B^-}{\Block_i,\Block_{i'}}(u))_{i'=1}^i$ is in order $\leq_0$ lexicographically before $(\nbfunc{\str B^-}{\Block_j,\Block_{j'}}(v))_{j'=1}^j$, where $\nbfunc{\str B^-}{\Block_i,\Block_{i}}(u)=u$ and analogously for $v$.
\end{enumerate}
 \item Finally put $u\leq_1 v$ if $u$ is original vertex and $v$ is ball vertex.
\end{enumerate}

It is easy to check that the order $\leq_1$ satisfies Definition~\ref{defn:lstar} and extends $\leq_\str A$ (the convex ordering is precisely the lexicographic order with respect to $\leq_0$ which we used in the construction).
\end{proof}

Let $\str A$ be an $L^\star_\Monoid$-structure whose every irreducible substructure belongs to $\mathcal M^\star_\Monoid$. We say that a ball vertex $b\in A$ representing a ball of diameter $\Block$ is an \emph{orphan} if there is no original vertex $v\in A$ such that $\func{A}{\Block}(v) = b$.

\begin{lemma}\label{lem:noorphans}
Let $\EMonoid$ be a distance monoid with finitely many blocks and let $\str A$ be an $L^\star_\Monoid$-structure such that every irreducible substructure of $\str A$ belongs to $\mathcal M^\star_\Monoid$. Then $\str A$ has a completion in $\mathcal M^\star_\Monoid$ if and only if $\str B$ does, where $\str B$ is an $L^\star_\Monoid$-structure which contains $\str A$ as a substructure such that moreover for every orphan $b\in A$ representing a ball of diameter $\Block$ we add an original vertex $v_b$ and ball vertices $b_{\Block'}$ to $\str B$ for every $\Block'\mlt \Block$ such that $\func{B}{\Block}(v_b) = b$ and $\func{B}{\Block'}(v_b) = b_{\Block'}$. We define all the other functions so that the closure of $v_b$ is from $\mathcal M^\star_\Monoid$.

Moreover, if $\str K\subseteq \str B$ has no completion in $\mathcal M^\star_\Monoid$ then neither does the substructure induced by $\str A$ on $K\cap A$.
\end{lemma}
\begin{proof}
Clearly, as $\str A\subseteq \str B$, every completion of $\str B$ is also a completion of $\str A$. We need to prove the other implication.

Suppose that $\str A$ has a completion $\str A'\in\mathcal M^\star_\Monoid$. By definition of $\mathcal M^\star_\Monoid$ there is a convexly ordered $\Monoid$-metric space $\str H$ such that $\str A'\subseteq L^\star(\str H)$. In particular, $L^\star(\str H)$ has no orphans. Using Proposition~\ref{prop:cycles}, one can then create $\str H'$ by ``duplicating'' some vertices of $\str H$ and thereby ensuring that every ball of $\str H$ has many sub-balls of smaller diameters. $L^\star(\str H')$ is then a completion of $\str B$.

The moreover part is in fact just the statement of the lemma said differently as we are adding original vertices to $\str B$ only because of local reasons. In other words, if we plug what $\str A$ induces on $K\cap A$ into the statement of this lemma as $\str A$, $\str B$ will precisely be $\str K$.
\end{proof}

From now on, when we are talking about completion $L^\star_\Monoid$-structures to $\mathcal M^\star_\Monoid$, we can assume that such structures have no orphans.

\medskip

Let $\EMonoid$ be a distance monoid with finitely many blocks, let $\str A$ be an $L^\star_\Monoid$-structure whose every irreducible substructure belongs to $\mathcal M^\star_\Monoid$ and let $u,u'$ be two ball vertices of $\str A$, both representing a ball of diameter $\Block$. If there is some $\ell\in\Monoid$ such that $(u,u')\in \rel{A}{t(\Block,\ell)}$ then for every $\Block'\mgt \Block$ it holds that either $\func{A}{\Block,\Block'}(u)=\func{A}{\Block,\Block'}(u')$, or $(\func{A}{\Block,\Block'}(u),\func{A}{\Block,\Block'}(u'))\in\rel{A}{t}$ where $t\supseteq t(\Block,\ell)$. (Note that this holds even in the corner case when $u,u'$ are original vertices and we understand their actual distance as $t(\Block_{\{0\}},d(u,u'))$.) This means that in such structures the information about ``block distances'' is consistent. Generalizing this observation, for every pair of vertices $u\neq v$ of such a structure and every block $\Block$ of $\Monoid$ we define $t(\Block, u,v)$ as follows:
$$t(\Block, u,v) = \begin{cases}
\{d_\str{A}(u,v)\} & \text{ if }\Block = \{0\}\text{ and }d_\str{A}(u,v)\text{ is defined,}\\
\Block & \text{ if }\func{A}{\Block}(u)=\func{A}{\Block}(v),\\
\Block & \text{ if }\Block\text{ is the maximal block of $\Monoid$},\\
t & \text{ if } (\func{A}{\Block}(u),\func{A}{\Block}(v)) \in \rel{A}{t},\\
\text{undefined} & \text{otherwise}.
\end{cases}$$

Using this, we define the \emph{block distance} $t(u,v)$ as $t(\Block,u,v)$ for the smallest block $\Block$ for which $t(\Block,u,v)$ is defined. Note that for every $u,v$ there is a block $\Block$ such that $t(u,v)\subseteq \Block$.

\begin{lemma}\label{lem:blockdistances}
Let $\EMonoid$ be a distance monoid with finitely many blocks and let $\str A$ be an $L^\star_\Monoid$-structure such that every irreducible substructure of $\str A$ belongs to $\mathcal M^\star_\Monoid$. Let $b\neq b'$ be ball vertices of $\str A$ representing balls of diameter $\Block$. If there exist original vertices $v,v'\in A$ such that $\func{A}{\Block}(v) = b$, $\func{A}{\Block}(v') = b'$ and $d_\str{A}(v,v') = \ell$, then $(b,b')\in\rel{A}{t(\Block,\ell)}$.
\end{lemma}
\begin{proof}
Follows from the fact that every irreducible substructure of $\str A$ belongs to $\mathcal M^\star_\Monoid$ where this holds.
\end{proof}

We now define an explicit completion procedure for $L^\star_\Monoid$-structures whose every irreducible substructure belongs to $\mathcal M^\star_\Monoid$. Let $\str A$ be an $L^\star_\Monoid$-structure whose every irreducible substructure belongs to $\mathcal M^\star_\Monoid$. For a pair of vertices $u,v\in A$ such that $t(u,v) = t(\Block,\ell)$, we say that $t(u,v)$ is \emph{witnessed} in $\str A$ if there are vertices $u',v'$ in $\str A$ such that $u\sim_\Block u'$, $v\sim_\Block v'$ and $d_\str{A}(u',v')\in t(u,v)$.

\begin{definition}\label{defn:starcompletion}
Let $\EMonoid$ be a distance monoid with finitely many blocks and let $\str A$ be an $L^\star_\Monoid$-structure satisfying the following:
\begin{enumerate}
	\item $\str A$ has no orphans,
	\item every irreducible substructure of $\str A$ belongs to $\mathcal M^\star_\Monoid$, and
	\item there is a linear extension $\leq_0$ of $\leq_\str A$.
\end{enumerate}

We define an $L^\star_\Monoid$-structure $\str C$ as follows:
\begin{enumerate}
\item Let $T$ be a finite set of block-types such that for every $u,v\in A$ we have $t(u,v)\in T$.
\item\label{starfcompletion:S} Let $S$ be a finite subset of $\Monoid$ such that $S$ contains a distance from each block, $\str A$ uses only distances from $S$ and for every $t\in T$ it holds that $S\cap t\neq \emptyset$.
\item For every $t\in T$ such that $t\subseteq \Block$ pick $b(t)\in t\cap S^\oplus$ to be either the largest element of $t(u,v)$ or $b(t)\mgeq \mus(\Block,S)$. Note that since $S$ contains a distance from each block and $S\cap t\neq \emptyset$, we can indeed pick $b(t)\in S^\oplus$.
\item Let $S' = S\cup \{b(t) : t\in T\}$ and for every $t=t(\Block,\ell)\in T$ pick $a(t)\in \Block\cap S^\oplus$ such that $a\mgeq \mus(\Block, S')$.
\item Let $\str G$ be the $\Monoid$-graph induced by $\str A$ on the set of original vertices.
\item\label{starfcompletion:distances} For every pair of original vertices $u,v$ of $\str A$ such that $t(u,v)$ is not witnessed in $\str G$, add a path with distances $a(t(u,v)),b(t(u,v)),\allowbreak a(t(u,v))$ into $\str G$ connecting $u$ and $v$ (the two interior vertices are new).
\item Let $\str G'$ be the $\Monoid$-shortest path completion of $\str G$ together with an arbitrary convex order (see Definition~\ref{defn:Js}) and put $\str C = L^\star(\str G')$.
\end{enumerate}
\end{definition}

\begin{lemma}\label{lem:starfcompletionworks}
In the setting of Definition~\ref{defn:starcompletion}, $\str G$ is $\Monoid$-metric if and only if $\str A$ has a completion in $\mathcal M^\star_\Monoid$. Moreover, if $\str A$ has a completion in $\mathcal M^\star_\Monoid$ then $\str C^-$ is a strong completion of $\str A^-$.
\end{lemma}
\begin{proof}
First assume that $\str G$ is $\Monoid$-metric. Since $\str G$ is $\Monoid$-metric, by Proposition~\ref{prop:cycles} it follows that $\str G'$ is a strong completion of $\str G$ (note that $\str G$ is connected thanks to putting $t(u,v)$ to be the largest block if otherwise it is undefined). Since, by the construction, every $t(u,v)$ is witnessed in $\str G$, it follows that $\str C^-$ is a strong completion of $\str A^-$ in $\mathcal M^\star_\Monoid$.

It remains to prove that if $\str G$ is non-$\Monoid$-metric then $\str A$ has no completion in $\mathcal M^\star_\Monoid$. For a contradiction assume that $\str A$ has a completion $\str A'\in \mathcal M^\star_\Monoid$ and that $\str G$ is non-$\Monoid$-metric. By Proposition~\ref{prop:cycles} we get that $\str G$ contains a non-$\Monoid$-metric cycle $\str K$. Some of the edges of $\str K$ are present in $\str A$, the rest come from the $a(t),b(t),a(t)$ paths. Let $\str K'$ be what $\str A'$ induces on $K\cap A$ (in other words, the completion of $\str K$ without the added paths). By our assumption $\str K'$ is an $\Monoid$-metric space.

Enumerate the endpoints of the $a(t),b(t),a(t)$ paths in $\str K$ as $u_1,v_1, u_2, \allowbreak v_2,\ldots,u_k,v_k$ and put $e_i = d_{\str A'}(u_i,v_i)$. Clearly, $e_i\in t(u_i,v_i)$. Note that by the definition of block-types and by the choice of $a(t)$ and $b(t)$, for every $1\leq i \leq k$ such that $t(u_i,v_i) = t(\Block,\ell)$ there is $a_i'\in \Block$ such that the triangle with distances $a_i', b(t(u_i,v_i)), e_i$ is $\Monoid$-metric. It follows that if we extend $\str K'$, adding a path $a_i',b(t(u_i,v_i)),a_i'$ connecting $u_i$ and $v_i$ for every $i$, the resulting graph is $\Monoid$-metric. This implies that if in Definition~\ref{defn:starcompletion} we used $a_i'$ instead of $a(t)$, we would obtain an $\Monoid$-metric graph $\str G$. However, this contradicts the choice of $a(t)$'s (see Lemma~\ref{lem:obstaclemus}) and we are done.
\end{proof}

\begin{lemma}\label{lem:finite4value2:2}
Let $\EMonoid$ be a distance monoid with finitely many blocks and let $\str{C}_0$ be a linearly ordered $L^\star_\Monoid$-structure. Then there is $n=n(\Monoid,\str{C}_0)$ such that the following holds: Let $\str{A}$ be an $L^\star_\Monoid$-structure such that there is a homomorphism-embedding $\str A\to\str C_{0}$ and moreover every irreducible substructure of $\str A$ is from $\mathcal M^\star_\Monoid$. If $\str A$ has no completion in $\mathcal M^\star_\Monoid$, then $\str A$ contains a substructure on at most $n$ vertices which has no completion in $\mathcal M^\star_\Monoid$.
\end{lemma}

% \marginpar{MK: Opravit Figure~\ref{fig:nonSmetric}.}
% \begin{figure}
% \centering
% \includegraphics{cyclecont.pdf}
% \caption{Consider monoid $\Monoid$ given by Example~\ref{example1} (\ref{ex:Smetric}) for $S=\{1,3,5\}$.
% The figure depicts the family of unimportant edges in a non-$\Monoid$-metric cycle (left), and the corresponding forbidden substructure in $\str{C}'$ (right).}
% \label{fig:nonSmetric}
% \end{figure}

\begin{proof}
Note that $\str A$ having a homomorphism-embedding to $\str C_0$ implies that there is a linear extension $\leq_0$ of $\leq_\str A$. Using Lemma~\ref{lem:noorphans} we can thus assume that $\str A$ contains no orphans and therefore it satisfies the conditions of Definition~\ref{defn:starcompletion}. By Lemma~\ref{lem:ordercompletion} we can also ignore the order.

We will use Definition~\ref{defn:starcompletion} together with Lemma~\ref{lem:starfcompletionworks} to study $\str G$ instead of $\str A$, which will allow us to use Proposition~\ref{prop:Sequivalence}. Also note that the sets $S$ and $T$ defined in part~\ref{starfcompletion:S} of Definition~\ref{defn:starcompletion} only depend on $\str C_0$ (because every distance and every block-type which appears in $\str A$ must also appear in $\str C_0$).

Put $$n = 2(m+1)n(S),$$
where $m$ is the number of blocks of $\Monoid$ and $n(S)$ is given by Proposition~\ref{prop:Sequivalence}.

% The main idea of the proof of this lemma is to reduce the possibly infinite family of forbidden cycles (given by Proposition~\ref{prop:cycles})
% to a finite family of forbidden substructures by use of the additional ball vertices which serve as shortcuts over unimportant parts of the cycle. An example is sketched
% in Figure~\ref{fig:nonSmetric}.

We will find a substructure $\str D$ of $\str A$ from Definition~\ref{defn:starcompletion} on at most $n$ vertcies which has no completion in $\mathcal M^\star_\Monoid$. By Lemma~\ref{lem:starfcompletionworks} we know that that $\str{G}$ is non-$\Monoid$-metric and thus, using Proposition~\ref{prop:cycles} we get a non-$\Monoid$-metric cycle $\str{K}$ in $\str{G}$. We can enumerate the edges in $\str{K}$ as $\ell, e_1, \ldots, e_k$ such that $\ell\mgt e_1\oplus e_2\oplus \cdots \oplus e_k$. Consider the family $(f_i)$ of important edges in $\str{K}$ given by Proposition~\ref{prop:Sequivalence}. 

Let $D'$ be the subset of vertices of $\str K$ such that if $f_i$ is $b(t(u,v))$ or $a(t(u,v))$ added in step~\ref{starfcompletion:distances} we put $u,v\in D'$ and if $f_i$ is an edge of $\str A$ we put its endpoints to $\str K$. Let $\str D$ be the closure of $D'$ in $\str B$ and let $\str P$ be a subpath of $\str K$ such that it has at least 3 vertices, its endpoints are in $D'$ and all its interior points are not in $D'$ (so in particular $\str P$ consists of unimportant edges). Since edges from the largest block are all important, it follows that the $\Monoid$-length of $\str P$ lies in a non-maximal block $\Block$ of $\Monoid$ and thus by Lemma~\ref{lem:finite4value2:1} there is a ball vertex $b$ such that for every $v\in \str P$ we have $\func{B}{\Block}(v) = b$. In other words, if there is a completion of $\str D$ in $\mathcal M^\star_\Monoid$ then the distance of the endpoints of $\str P$ will lie in block $\Block$. By Proposition~\ref{prop:Sequivalence} it follows that $\str D$ has no completion in $\mathcal M^\star_\Monoid$.

\end{proof}

\begin{lemma}
\label{lem:finite4value2}
Let $\EMonoid$ be a distance monoid with finitely many blocks.  Then 
$\mathcal M^\star_\Monoid$  is a Ramsey class.
\end{lemma}
\begin{proof}
We first show that $\mathcal M^\star_\Monoid$ is an $\mathcal R_\Monoid$-multiamalgamation class where
the class $\mathcal R_\Monoid$ consists of all finite $L^\star_\Monoid$-structures (which is Ramsey by Theorem~\ref{thm:NRclosures}).

Clearly $\mathcal M^\star_\Monoid$ is hereditary. It is also a strong amalgamation class: To see that it is enough to show that free amalgamations of structures from $\mathcal M^\star_\Monoid$ have a completion in $\mathcal M^\star_\Monoid$. Let $\str A,\str B_1,\str B_2\in \mathcal M^\star_\Monoid$ be structures such that $\str A\subseteq \str B_1,\str B_2$ and consider the amalgamation problem for $\str B_1$ and $\str B_2$ over $\str A$. Using Lemma~\ref{lem:ordercompletion} we know that it is enough to discuss completion of the unordered reduct. We can also assume that $\str B_1$ and $\str B_2$ have no orphans (and thus the free amalgamation $\str C$ of $\str B_1$ and $\str B_2$ over $\str A$ also has no orphans).

Using Lemma~\ref{lem:starfcompletionworks} for $\str C$ we reduce the question whether $\str C$ has a completion in $\mathcal M^\star_\Monoid$ to the question whether the corresponding $\Monoid$-graph $\str G$ is $\Monoid$-metric. If $\str G$ is non-$\Monoid$-metric, and we take $\str G$ with smallest number of vertices, it follows that it has at most 4 vertices and it is easy to verify that the $\rel{}{t(\Block,\ell)}$ relations prevent this from happening in free amalgamations. Hence $\mathcal M^\star_\Monoid$ is a strong amalgamation class.

The locally finite completion property (see Definition~\ref{def:multiamalgamation}) follows by Lemma~\ref{lem:finite4value2:2}: Fix $\str B\in \mathcal M^\star_\Monoid$ and $\str C_0\in \mathcal R_\Monoid$ and let $\str C$ be an $L^\star_\Monoid$-structure such that there is a homomorphism-embedding $\str C\to\str C_0$ and every substructure of $\str C$ on at most $n=n(\Monoid,\str{C}_0)$ (given by Lemma~\ref{lem:finite4value2:2}) has a completion in $\mathcal M^\star_\Monoid$.  We want to show that $\str C$ has a completion in $\mathcal M^\star_\Monoid$ with respect to copies of $\str B$.

Without loss of generality we can assume that every vertex and every tuple in every relation and function in $\str C$ lies in a copy of $\str B$ (because otherwise we can simply remove it, not changing the existence of a completion with respect to copies of $\str B$). This, however, implies that in particular every irreducible substructure of $\str C$ is from $\mathcal M^\star_\Monoid$. Hence the conditions of Lemma~\ref{lem:finite4value2:2} are satisfied for $\str C$, which implies that if $\str C$ has no completion if $\mathcal M^\star_\Monoid$ then it contains a substructure on at most $n$ vertices with no completion in $\mathcal M^\star_\Monoid$. Consequently, $\str C$ has a completion in $\mathcal M^\star_\Monoid$ and we are done.
\end{proof}

Now we can prove the main theorem of this paper:
\begin{proof}[Proof of Theorem~\ref{thm:main}]
To show that $\vv{\mathcal M}_\Monoid$ is Ramsey consider $\str{A},\str{B}\in \vv{\mathcal M}_\Monoid$. We show that there exists $\str{C}\longrightarrow(\str{B})^\str{A}_2$.  If $\Monoid$ has only finitely many blocks, then by Lemma~\ref{lem:finite4value2} there exists $\str C'\in \mathcal M^\star_\Monoid$ such that
$\str{C}'\longrightarrow(L^\star(\str{B}))^{L^\star(\str{A})}_2$. In fact, we can pick $\str C'$ such that there is $\str C\in \vv{\mathcal M}_\Monoid$ and $\str C' = L^\star(\str C')$ (by the moreover part of Lemma~\ref{lem:starfcompletionworks} and by Lemma~\ref{lem:ordercompletion}). Because the functor $L^\star$ preserves substructures, it follows that $\str C\longrightarrow (\str B)^\str A_2$.

Now let $\Monoid$ be an arbitrary distance monoid and let $\str{B} \in \vv{\mathcal M}_\Monoid$. We will observe that $\str{B}$ lies in an amalgamation subclass of $\vv{\mathcal M}_\Monoid$ which happens to be Ramsey. Therefore for every $\str{A}$ which is a substructure of $\str{B}$ and every integer $k$, we get a $\str{C}$ from the subclass (and hence from $\vv{\mathcal M}_\Monoid$) such that $\str{C} \rightarrow \left(\str{B}\right)^\str{A}_k$, thus proving the Ramsey property for $\vv{\mathcal M}_\Monoid$.

And finding the subclass is actually easy. Let $S$ be the set of all distances which occur in $\str{B}$. Clearly $S$ is finite. Let $\langle S\rangle$ be the submonoid of $\Monoid$ generated by $S$ (containing every finite sum of elements from $S$, hence $\langle S \rangle = S^\oplus\cup \{0\}$ as defined in Section~\ref{subsec:nonimportant}). It is easy to check that if we denote by $\oplus'$ and $\mleq'$ the restriction of $\oplus$ and $\mleq$ respectively to $\langle S\rangle$, then $(\langle S\rangle, \oplus', \mleq', 0)$ is a distance monoid with finitely many blocks and hence $\str{B} \in \vv{\mathcal M}_{\langle S\rangle}$. From this point we can proceed by application of Lemma~\ref{lem:finite4value2} as in the first case.
\end{proof}

\section{Conclusion}
\label{sec:remarks}

The techniques introduced in this paper can be generalised and used to prove the
Ramsey property for much wider family of $\Monoid$-valued metric spaces that
contains everything discussed in this paper as well as for example
$\Lambda$-ultrametric spaces, where $\Lambda$ is a finite distributive lattice.
(These spaces were introduced and their Ramsey expansions were found by
Braunfeld~\cite{Sam}.) We can also prove EPPA for all those classes (hence
extending the results and answering a question of Conant~\cite{Conant2015}).
This can be done by combining the Herwig-Lascar theorem~\cite{herwig2000} with 
additional unary functions as done in~\cite{Evans3}.
Finally, the shortest path completion can be utilised to obtain a Stationary
Independence Relation for the $\Monoid$-valued metric spaces~\cite{Tent2013}.  Our work was
also motivated by the analysis of Cherlin's catalogue of metrically homogeneous
graphs~\cite{Cherlin2013} done in~\cite{Aranda2017,Aranda2017a,Aranda2017c,Konecny2018bc} which also
gives rise to monoid-valued metric spaces, however this correspondence is
more subtle (the monoids do not satisfy the notion of distance monoid used in
this paper).

These results will appear in~\cite{Hubicka2017sauer}, see also~\cite{Konecny2018b}.

\subsection{Acknowledgment}
We would like to thank the anonymous referee for remarks that improved
the quality of this paper.

\bibliographystyle{amsplain}
\bibliography{ramsey.bib}
\end{document}